\documentclass{amsproc}
\usepackage{amssymb}
\usepackage{amscd}
\usepackage{xypic}
\usepackage{amsmath}
\usepackage{amsfonts}
\usepackage[dvips]{graphicx}
\usepackage{hyperref}
\usepackage[arrow,curve,matrix]{xy}

\newtheorem{theorem}{Theorem}[section]
\newtheorem{lemma}[theorem]{Lemma}
\newtheorem{conjecture}[theorem]{Conjecture}
\newtheorem{corollary}[theorem]{Corollary}
\newtheorem{proposition}[theorem]{Proposition}

\theoremstyle{definition}
\newtheorem{definition}[theorem]{Definition}

\theoremstyle{remark}
\newtheorem{remark}[theorem]{Remark}

\numberwithin{equation}{section}

\def\PP{{\textbf P}}
\def\OO{\mathcal{O}}

\def\F{\mathcal{F}}

\def\K{\mathcal{K}}

\def\cM{\mathcal{M}}
\def\cR{\mathcal{R}}
\def\rr{\overline{\mathcal{R}}}

\def\cU{\mathcal{U}}

\def\mm{\overline{\mathcal{M}}}

\title{Green's conjecture for general covers}
\author[M. Aprodu]{Marian Aprodu}
\address{Institute of Mathematics "Simion Stoilow" of the Romanian
Academy, RO-014700 Bucharest -- Romania,
and \c Scoala Normal\u a Superioar\u a Bucure\c sti,
Calea Grivi\c tei 21, Sector 1, RO-010702 Bucharest, Romania}
 \email{{\tt marian.aprodu@imar.ro}}
\thanks{M. Aprodu was supported in part by the SFB 647 Raum-Zeit-Materie, by
a PN-II-ID-PCE-2008-2 grant, PCE-2228 (contract no. 502/2009) and
by LEA MathMode. He warmly thanks HU Berlin, IEC Nancy and Institut Poincar\'e Paris
for hospitality.}

\author[G. Farkas]{Gavril Farkas}
\address{Humboldt-Universit\"at zu Berlin, Institut f\"ur Mathematik, Unter den Linden 6,
10099 Berlin, Germany}
\email{{\tt farkas@math.hu-berlin.de}}
\thanks{G. Farkas was supported by the SFB 647 Raum-Zeit-Materie
and by the Schwerpunktprogramm of the DFG Algorithmische Methoden in Algebra und Zahlentheorie.}


\begin{document}

\maketitle

\section{Introduction}

M. Green's Conjecture on  syzygies of canonical curves
$\phi_{K_C}:C\rightarrow \PP^{g-1}$, asserting the following
vanishing of Koszul cohomology groups \cite{Green-JDG84}
$$K_{p, 2}(C, K_C)=0\Leftrightarrow p<\mathrm{Cliff}(C),$$
has been one of the most investigated problems in the last decades
in the theory of algebraic curves. Based on the principle that all
non-trivial syzygies are generated by secants to the canonical curve
$C\subset \PP^{g-1}$, the conjecture is appealing because it predicts
that one can read off the Clifford index of the curve (measuring the
complexity of $C$ in its moduli space) from the graded Betti diagram
of the canonical embedding. Voisin \cite{Voisin-even}, \cite{Voisin-odd}
established Green's conjecture for general curves $[C]\in \cM_g$ of any genus.
\vskip 4pt
Building on the work of Voisin, the first author \cite{A-MRL} has found a Brill-Noether theoretic sufficient condition for a curve to satisfy Green's Conjecture. If  $[C]\in \cM_g$ is a $d$-gonal curve with $2\leq d\leq \frac{g}{2}+1$ satisfying the \emph{linear growth condition}
\begin{equation}\label{lgc}
\mathrm{dim }\  W^1_{g-d+2}(C)=\rho(g,1,g-d+2) = g-2d+2,
\end{equation}
then $C$ satisfies both Green's Conjecture and the Gonality Conjecture \cite{GL-normality}.
\vskip 3pt

Condition (\ref{lgc}) is equivalent to $\mathrm{dim }\  W^1_{d+n}(C)\leq n$
for all $0\le n\le g-2d+2$. In particular, it implies that $C$ has a finite number of pencils of minimal degree. The case of odd genus
and maximal gonality treated by \cite{Voisin-odd} is automatically excluded from condition (\ref{lgc}). One aim of this paper is to establish
Green's conjecture for classes of curves where condition (\ref{lgc}) manifestly fails, in particular for curves having an infinite number of minimal pencils. Typical examples are curves whose Clifford indices are not computed
by pencils, and
their covers. Precisely, if $X$ is a curve of Clifford dimension $r(X):=r\geq 2$, then $\mbox{gon}(X)=\mbox{Cliff}(X)+3$ and $X$ carries an infinite number of pencils of minimal degree \cite{CM}.
If $f:C\rightarrow X$ is a branched covering of $X$ of sufficiently high genus, then  $\mbox{gon}(C)=\mbox{deg}(f)\cdot \mbox{gon}(X)$ and $C$
carries infinitely many pencils of minimal degree, all pulled-back from $X$. In particular, condition (\ref{lgc}) fails for $C$.

\begin{theorem}\label{doubleplane}\hfill

(i) Set $d\geq 3,\ g\geq d^2+1$ and let $C\rightarrow \Gamma\subset \PP^2$ be a general genus $g$ double covering of a smooth plane curve of
degree $d$. Then $K_{2d-5, 2}(C, K_C)=0$ and $C$ satisfies Green's Conjecture.

(ii) Let $g\geq 2d^2+1$ and $C\rightarrow \Gamma\subset \PP^2$ be a general genus $g$ fourfold cover of a smooth plane curve of degree $d$. Then $C$
satisfies Green's Conjecture.
\end{theorem}

In a similar vein, we have a result about triple coverings of elliptic curves.
\begin{theorem}
\label{thm: triple}
 Let $C\to E$ be a general triple covering of genus $g\geq 13$ of an elliptic curve. Then $K_{3,2}(C,K_C)=0$ and $C$ satisfies Green's Conjecture.
\end{theorem}
Curves with Clifford dimension $3$ have been classified in \cite{ELMS}. If $[X]\in \cM_g$ is such that $r(X)=3$, then $g=10$ and $X$ is the complete intersection
of two cubic surfaces in $\PP^3$. The very ample $\mathfrak g^3_9$ computes $\mathrm{Cliff}(X)=3$, whereas $\mbox{dim }W^1_6(C)=1$; each minimal pencil of
$X$ is induced by planes through a trisecant line to $X\subset \PP^3$. We prove the following result:
\begin{theorem}\label{cliffdim3}
Let $C\to X$ be a general double covering of genus $g\geq 28$ of a smooth curve $X$ with $r(X)=3$. Then $K_{9, 2}(C, K_C)=0$ and $C$ satisfies Green's Conjecture.
\end{theorem}

\vskip 5pt

The second aim of this paper is to study syzygies of curves with a fixed point free involution. We denote by $\cR_g$ the moduli space of pairs $[C, \eta]$ where $[C]\in \cM_g$ and $\eta\in \mbox{Pic}^0(C)-\{\OO_C\}$ is a root of the trivial bundle, that is,  $\eta^{\otimes 2}=\OO_C$. Equivalently, $\cR_g$ parametrizes \'etale double covers of curves $f:\widetilde{C}\rightarrow C$, where $g(\widetilde{C})=2g-1$ and $f_*(\OO_{\widetilde{C}})=\OO_C\oplus \eta.$ The moduli space $\cR_g$ admits a Deligne-Mumford compactification $\rr_g$ by means of stable Prym curves, that comes equipped with two morphisms
$$\pi:\rr_g\rightarrow \mm_g \ \ \ \mbox{ and }\ \ \  \chi:\rr_g\rightarrow \mm_{2g-1},$$
obtained by forgetting $\widetilde{C}$ and $C$ respectively. We refer to \cite{FL} for a detailed study of the birational geometry and intersection theory of $\rr_g$.
\vskip 3pt

One may ask whether Green's conjecture holds for a curve $[\widetilde{C}]\in \cM_{2g-1}$ corresponding to a general point
$[\widetilde{C}\stackrel{f}\rightarrow C]\in \cR_g$. Note that since $\widetilde{C}$ does not satisfy Petri's theorem
\footnote{Choose an odd theta-characteristic $\epsilon \in \mbox{Pic}^{g-1}(C)$ such that $h^0(C, \eta\otimes \epsilon)\geq 1$. Then $f^*(\epsilon)$ is a theta-characteristic on $\widetilde{C}$ with
$h^0(\widetilde{C}, f^*(\epsilon))=h^0(C, \epsilon)+h^0(C, \epsilon\otimes \eta)\geq 2,$
that is, $\widetilde{C}$ possesses a vanishing theta-null.}
the question is a little delicate. In spite of this fact we have the following answer:

\begin{theorem}\label{greencov}  Let us fix a general \'etale double cover $[f:\widetilde{C}\rightarrow C]\in \cR_g$.

(i) If $g\equiv 1 \ \mathrm{mod}\  2$, then $\widetilde{C}$ is of maximal gonality, that is, $\mathrm{gon}(\widetilde{C})=g+1$. In particular
 $\widetilde{C}$ satisfies Green's Conjecture.

(ii) If $g\equiv 0 \ \mathrm{ mod }\  2$, then $\mathrm{gon}(\widetilde{C})=g$.
\end{theorem}

We provide two proofs that $\widetilde{C}$ satisfies Green's conjecture for odd $g$. The statement concerning Green's conjecture follows via condition (\ref{lgc}).
As we explain in what follows, we also establish that $\widetilde{C}$ satisfies Green's conjecture for even $g$.
The proof, which we now briefly explain
can be viewed as a counterpart of Voisin's result \cite{Voisin-even} and has the advantage of singling out an explicit locus in $\cR_g$ where Green's conjecture holds.
\vskip 4pt

Let $\F_g^{\mathfrak{N}}$ be the $11$-dimensional moduli space of genus $g$ Nikulin surfaces. A very general\footnote{That is, a point outside a countable union of (Noether-Lefschetz) divisors on $\F_g^{\mathfrak{N}}$.} point
of
$\F_g^{\mathfrak{N}}$ corresponds to a double cover $f:\widetilde{S}\rightarrow S$ of a $K3$ surface, branched
along a set $R_1+\cdots+R_8$ of eight mutually disjoint $(-2)$-curves, as well as a linear system $L\in \mbox{Pic}(S)$,
where $L^2=2g-2$ and $L\cdot R_i=0$ for $i=1, \ldots, 8$. We choose a smooth curve $C\in |L|$,
set $\widetilde{C}:=f^{-1}(C)\subset \widetilde{S}$. Then
the restriction $f_C:\widetilde{C}\rightarrow C$ defines an element of $\cR_g$.
We show that the canonical bundle of $\widetilde{C}$ has minimal syzygies when the lattice $\mbox{Pic}(S)$ is minimal, that is, of rank $9$.

\begin{theorem}\label{nik1}
Let $f_C:\widetilde{C}\rightarrow C$ be a double cover corresponding to a very general Nikulin surface of genus $g$.

(i) If $g\equiv 1 \ \mathrm{mod}\  2$, then $\mathrm{gon}(\widetilde{C})=g+1$.

(ii) If $g\equiv 0 \ \mathrm{ mod }\  2$, then $\mathrm{gon}(\widetilde{C})=g$.

\noindent In both cases, the curve
 $\widetilde{C}$ verifies Green's Conjecture.
\end{theorem}
We point out that in this situation both $C$ and $\widetilde{C}$ are sections of (different) $K3$ surfaces,
hence by \cite{AF-Compositio} they verify Green's Conjecture. The significance of Theorem \ref{nik1} lies
 in showing that the Brill-Noether theory of $C$ and $\widetilde{C}$ is the one expected from a general element of $\cR_g$.

\section{Koszul cohomology}
\label{sec: Koszul}

We fix a smooth algebraic curve $C$, a line bundle
$L$ on $C$ and a space of sections $W\subset H^0(C, L)$. Given
two integers $p$ and $q$, the Koszul cohomology group  $K_{p,q}(C,L,W)$ is the cohomology at the middle of the complex
$$
\wedge^{p+1}W\otimes H^0(C,L^{\otimes(q-1)})\longrightarrow
\wedge^{p}W\otimes H^0(C,L^{\otimes q})\longrightarrow
\wedge^{p-1}W\otimes H^0(C,L^{\otimes(q+1)})
$$
If $W=H^0(C,L)$ we denote the corresponding Koszul cohomology group
by $K_{p,q}(C,L)$.

\medskip

For a globally generated line bundle $L$,  Lazarsfeld \cite{Lazarsfeld-ICTP} provided a
description of Koszul cohomology in terms of kernel bundles.
If $W\subset H^0(C,L)$ generates $L$ one defines
$
M_W:=\mathrm{Ker}\{W\otimes\OO_C\to L\}.
$
When $W = H^0(C,L)$, we write $M_W := M_L$.
The kernel of the Koszul differential
coincides with
\begin{eqnarray*}
H^0(C,\wedge^p M_W\otimes L^{q})\subset \wedge^{p}W\otimes H^0(C,L^{\otimes q})
\end{eqnarray*}
and hence one has the following isomorphism:

\begin{eqnarray*}
K_{p,q}(C,L,W) & \cong &
\mathrm{Coker}\Bigl\{\wedge^{p+1}W\otimes H^0(C,L^{q-1})\to H^0(C,\wedge^p M_W
\otimes L^{q})\Bigr\} .
\end{eqnarray*}

Note that for $q=1$ the hypothesis of being globally generated
is no longer necessary, and we do have a similar description
for $K_{p,1}$ with values in any line bundle. Indeed, if $\mathrm{Bs}|L|=B$,
and $M_L$ is the kernel of the evaluation map on global sections, then
$M_L\cong M_{L(-B)}$. Applying the definition, the identification $H^0(C,L(-B))\cong H^0(C,L)$
and the inclusion $H^0(C,L(-B)^{\otimes 2})\subset H^0(C,L^{\otimes 2})$
induce an isomorphism, for any $p$, between $H^0(C,\wedge^p M_{L(-B)}
\otimes L(-B))$ and $H^0(C,\wedge^p M_L\otimes L)$.
In particular,  $K_{p,1}(C,L(-B))\cong K_{p,1}(C,L)$ and
\begin{eqnarray*}
K_{p,1}(C,L) & \cong &
\mathrm{Coker}\Bigl\{\wedge^{p+1}H^0(C,L)\to H^0(C,\wedge^p M_L
\otimes L)\Bigr\}
\end{eqnarray*}
as claimed.

\subsection{Projections of syzygies}

Let $L$ be a line bundle
on $C$ and assume that $x\in C$ is not
a base point of $L$.
Setting $W_x:=H^0(C, L(-x))$, we have
an induced short exact sequence
$$
0\longrightarrow W_x\longrightarrow H^0(C,L)\longrightarrow
{\mathbb C}_x\longrightarrow 0.
$$
From the restricted Euler
sequences corresponding to $L$ and $L(-x)$ respectively,
we obtain an exact sequence
$$
0\longrightarrow M_{L(-x)}\longrightarrow M_L\longrightarrow {\mathcal O}_C(-x)\longrightarrow 0,
$$
and further, for any integer $p\geq 0$,
$$
0\longrightarrow \wedge^{p+1} M_{L(-x)}\otimes L
\longrightarrow \wedge^{p+1}M_L\otimes L\longrightarrow
\wedge^pM_{L(-x)}\otimes L(-x).
$$
The exact sequence
of global sections, together with the natural sequence
$$
0\longrightarrow \wedge^{p+2}W_x\longrightarrow
\wedge^{p+2}H^0(C,L)\longrightarrow
\wedge^{p+1}W_x\longrightarrow 0,
$$
induce an exact sequence
$$
0\rightarrow K_{p+1,1}(C,L,W_x)\longrightarrow
K_{p+1,1}(C,L)\stackrel{\mathrm{pr_x}}\longrightarrow K_{p,1}(C,L(-x)),
$$
where the induced map $\mbox{pr}_x:K_{p+1,1}(C,L){\rightarrow}K_{p,1}(C,L(-x))$
is the {\em projection of syzygies} map centered at $x$. Nonzero Koszul classes survive when they
are projected from general points:

\begin{proposition}
If $0\ne \alpha\in K_{p+1,1}(C,L)$, then $\mathrm{pr}_x(\alpha)\ne 0\in K_{p, 1}(C, L(-x))$ for a general point $x\in C$.
\end{proposition}

We record some immediate consequences and refer to \cite{A-MZ02} for complete proofs
based on semicontinuity.

\begin{corollary}
\label{cor: projection}
Let $L$ be a line bundle on a curve $C$ and $x\in C$ a
point.
If $L(-x)$ is nonspecial and $K_{p,1}(C,L(-x))=0$ then $K_{p+1,1}(C,L)=0$.
\end{corollary}

Going upwards, it follows from Corollary \ref{cor: projection}
that, for a nonspecial $L$,
the vanishing of $K_{p,1}(C,L)$ implies that $K_{p+e,1}(C,L(E))=0$, for any effective
divisor $E$ of degree $e$.

For canonical nodal curves, we have a similar result:

\begin{corollary}
\label{cor: projection can}
Let $L$ be a line bundle on a curve $C$ and $x,y\in C$ two
points.
If $K_{p,1}(C,K_C)=0$ then $K_{p+1,1}(C,K_C(x+y))=0$.
\end{corollary}

The proof of Corollary \ref{cor: projection can} follows directly from the
Corollary \ref{cor: projection} for $L=K_C(x+y)$ coupled with isomorphisms
$K_{p,1}(C,K_C(y))\cong K_{p,1}(C,K_C)$. Geometrically, the image
of $C$ under the linear system $|K_C(x+y)|$ is a nodal canonical curve,
having the two points $x$ and $y$ identified, and the statement corresponds
to the projection map from the node.

By induction, from Corollary \ref{cor: projection can} and \ref{cor: projection} we obtain:

\begin{corollary}
\label{cor: canonical}
Let $C$ be a curve and $p\geq 1$  such that
$K_{p,1}(C,K_C)=0$. Then for any effective divisor
$E$ of degree $e$, we have $K_{p+e-1,1}(C,K_C(E))=0$.
\end{corollary}

\subsection{Koszul vanishing}

Using a secant construction, Green and Lazarsfeld \cite{GL-nonvanishing} have shown that non-trivial geometry (in the forms of existence of special linear series) implies non-trivial syzygies. Precisely,
if $C$ is a curve of genus $g$ and $\mathrm{Cliff}(C)=c$, then
$
K_{g-c-2,1}(C,K_C)\ne 0,
$
or equivalently, by duality, $K_{c,2}(C,K_C)\ne 0$.
Green \cite{Green-JDG84} conjectured that this should be optimal and the converse should hold:

\begin{conjecture}
\label{conj: Green}
For any curve $C$ of genus $g$ and Clifford index
$c$, one has that
$$
K_{g-c-1,1}(C,K_C) = 0,
$$
equivalently, $K_{p,2}(C,K_C)=0$ for all $p<c$.
\end{conjecture}

\medskip

In the case of a nonspecial line bundle  $L$
on a curve $C$ of gonality $d$, \cite{GL-nonvanishing} gives us the non-vanishing of
$K_{h^0(L)-d-1,1}(C,L)\neq 0$.
In the same spirit, one may ask whether this result  is optimal.
It was conjectured in \cite{GL-normality} that
this should be the case for bundles of large degree.

\begin{conjecture}
\label{con: gonality}
For any curve $C$ of gonality $d$ there exists
a nonspecial very ample line bundle $L$ such that
$K_{h^0(L)-d,1}(C,L)=0$.
\end{conjecture}

\subsection{Curves on $K3$ surfaces}
It was known since the eighties that the locus
$$\mathcal{K}_g:=\{[C]\in \cM_g: C\mbox{ lies on a } K3 \mbox{ surface}\}$$
does not lie in any proper Brill-Noether stratum in $\cM_g$.
Most notably, curves $[C]\in \mathcal{K}_g$ lying on $K3$ surfaces $S$ with $\mathrm{Pic}(S)=\mathbb Z\cdot C$ satisfy the Brill-Noether-Petri theorem, see \cite{Lazarsfeld-K3}. This
provides a very elegant solution to the Petri conjecture, and remains to this day, the only explicit example of a \emph{smooth} Brill-Noether general curve of unbounded genus.

Green's hyperplane
section theorem \cite{Green-JDG84} asserts that
the Koszul cohomology of any $K3$ surface is isomorphic
to that of any hyperplane section, that is, 
$$K_{p, q}(S, \OO_S(C))\cong K_{p, q}(C, K_C).$$
Voisin has used this fact to find a solution
to Green's  conjecture for generic curves, see \cite{Voisin-even}, \cite{Voisin-odd}:

\begin{theorem}
Let $C$ be a smooth curve lying on a $K3$ surface $S$ with $\mathrm{Pic}(S)=\mathbb Z\cdot C$. Then
$C$ satisfies Green's conjecture.
\end{theorem}

This result has been extended in \cite{AF-Compositio} to cover the case of $K3$ surfaces with arbitrary Picard lattice, in particular curves with arbitrary gonality:

\begin{theorem}
\label{thm: AF}
Green's conjecture is valid for any smooth curve $[C]\in \K_g$
of genus $g$ and gonality $d\le [\frac{g}{2}]+1$.
The gonality conjecture is valid for smooth curves of Clifford
dimension one on a $K3$ surface, general in their linear systems.
\end{theorem}

\medskip

It is natural to ask whether in a linear system whose smooth
members are of Clifford dimension one the condition
(\ref{lgc}) is preserved. The answer in NO, as we shall see
in section \ref{sec: 6-gonal}.

\section{Syzygy conjectures for general \'etale double covers}
In this section we prove Theorem \ref{greencov} by degeneration. We begin by observing that if $g=2i$ with $i\in \mathbb Z_{>0}$ and
 $f:\widetilde{C}\rightarrow C$ is an \'etale double cover
of the genus-$g$ curve $C$, with $f_*\OO_{\widetilde{C}}=\OO_C\oplus \eta$, then $\widetilde{C}$ cannot possibly
 have maximal Clifford index (gonality). The difference variety $C_i-C_i\subset \mbox{Pic}^0(C)$ covers the Jacobian  $\mbox{Pic}^0(C)$ and there
 exist effective divisors $D, E\in C_i$ such that $\eta=\OO_C(D-E)$. We set $A:=f^*(\OO_C(E))\in \mbox{Pic}^g(\widetilde{C})$ and note that
$$h^0(\widetilde{C}, A)=h^0(C, f_*f^*(\OO_C(E))=h^0(C, \OO_C(E))+h^0(C, \OO_C(D))\geq 2, $$
that is, $A\in W^1_g(C)$. This shows that the image of the map $$\chi:\rr_g\rightarrow \mm_{2g-1}, \ \ \ \chi\bigl([\widetilde{C}\stackrel{f}
\rightarrow C]\bigr):= [\widetilde{C}]$$ is contained in the Hurwitz divisor $\mm_{2g-1, g}^1\subset \mm_{2g-1}$ of curves with
a $\mathfrak g^1_g$. For odd $g$ there is no obvious reason why $\widetilde{C}$ should have non-maximal gonality and  indeed,
we shall show that $\mbox{gon}(\widetilde{C})=g+1$ in this case.
\vskip 3pt

To prove Theorem \ref{greencov} we use the following degeneration. Fix a general pointed curve $[C, p]\in \cM_{g-1, 1}$ as well as an elliptic curve $[E, p]\in \cM_{1, 1}$. We fix a non-trivial point $\eta_E\in \mbox{Pic}^0(E)[2]$, inducing an \'etale double cover $f_E:\widetilde{E}\rightarrow E$, and set $\{x, y\}:=f_E^{-1}(p)$. The points $x, y\in \widetilde{E}$ satisfy the linear equivalence $2x\equiv 2y$. We choose two identical copies $(C_1, p_1)$ and $(C_2, p_2)$ of $(C, p)$ and consider the stable curve of genus $2g-1$
$$X_g:=C_1\cup E\cup C_2/p_1\sim x, p_2\sim y,$$
admitting an admissible double cover $f:X_g \rightarrow C\cup_p E$, which can be viewed as a point in the boundary divisor $\pi^*(\Delta_1)\subset \rr_g$. Note that $f$ maps both copies $(C_i, p_i)$ isomorphically onto $(C, p)$.

\begin{figure}[htb!]
\centering%
\includegraphics[width=7.5cm, height=4.5cm]{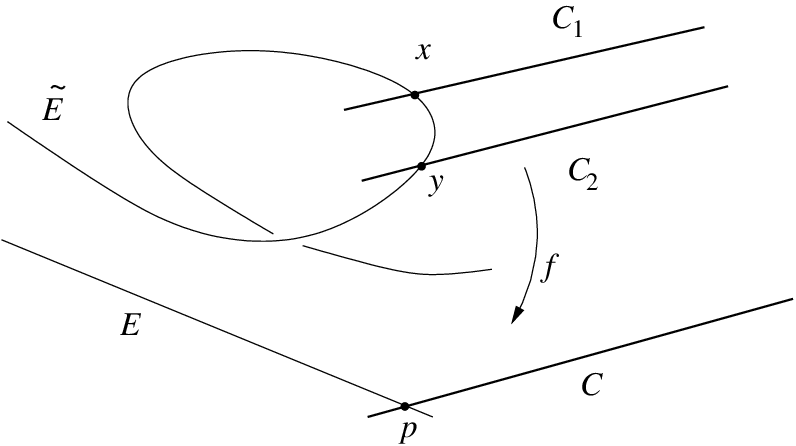}
\caption{The curve $X_g$.}
\label{fig: Xg}
\end{figure}

Theorem \ref{greencov} follows from the following computation coupled with an application of \cite{A-MRL}. Throughout the following proof we use the notation of \cite{EH} and
assume some familiarity with the theory of limit linear series. In particular, we recall that if $l\in G^r_d(C)$ is a linear series on a smooth curve $C$ of genus $g$ and $p\in C$, then one defines the \emph{adjusted Brill-Noether number} $\rho(l):=\rho(g, r, d)-w^l(p)$, where $w^l(p)$ is the \emph{weight} of the point $p$ with respect to $l$.

\begin{proposition} Let $[X_g\stackrel{f}\rightarrow C\cup E]\in \rr_{g}$ be the cover constructed above.\hfill

(i) If $g$ is odd then $\mathrm{gon}(X_g)=g+1$, that is, $[X_g]\notin \mm_{2g-1, g}^1$.

(ii) If $g$ is even then $\mathrm{gon}(X_g)=g$.
\end{proposition}
\begin{proof}
Suppose first that $X_g$ possesses a limit linear series $l\in \overline{G}^1_g(X_g)$ and
denote by $l_{C_1}, l_{C_2}$ and $l_{\widetilde{E}}$ respectively, its aspects on the
components of $X_g$.
From the additivity of the adjusted Brill-Noether number (obtained
by subtracting the ramification indices from the classical Brill-Noether number)
we obtain that
\begin{equation}\label{ineq1}
-1=\rho(2g-1, 1, g)\geq \rho(l_{C_1}, p_1)+\rho(l_{C_2}, p_2)+\rho(l_{\widetilde{E}}, x, y).
\end{equation}

Furthermore $\rho(l_{C_i}, p_i)\geq 0$,  because $[C_i, p_i]\in \cM_{g-1, 1}$ is general
and we apply \cite{EH} Theorem 1.1.
Using that the two points are generic,
it is easy to prove that $\rho(l_{\widetilde{E}}, x, y)\geq -1$. This shows that one has equality in (\ref{ineq1}), that is, $l$ is a \emph{refined limit} $\mathfrak g^1_g$ and moreover $\rho(l_{C_i}, p_i)=0$ for $i=1, 2$, and $\rho(l_{\widetilde{E}}, x, y)=-1$. We denote by $(a_0, a_1)$ (respectively $(b_0, b_1)$) the vanishing sequence of $l_{\widetilde{C}}$ at the point $x$ (respectively $y$). From the compatibility of vanishing sequences at the nodes $x$ and $y$, we find that $a_0+a_1=g$ and $b_0+b_1=g$ respectively.
On the other hand $l_{\widetilde{E}}$ possesses a section which vanishes at least with order $a_0$ at $x$ as well as with order $b_1$ at $y$ (respectively a section which vanishes at least with order $a_1$ at $x$ and order $b_0$ at $y$). Therefore $a_0+b_1\leq g$ and $a_1+b_0\leq g$. All in all, since $\rho(l_{\widetilde{E}}, x, y)=-1$, this implies that $a_0=b_0$ and $a_1=b_1=g-a_0$, and the following linear equivalence on $\widetilde{E}$ must hold:
$$a_0\cdot x+(g-a_0)\cdot y\equiv (g-a_0)\cdot x+a_0\cdot y.$$
Since $x-y\in \mathrm{Pic}^0(\widetilde{E})[2]$, we obtain that $g-2a_0 \equiv \mbox{ mod } 2$. When $g$ is odd this yields a
contradiction. On the other hand when $g$ is even, this argument shows that $\mbox{gon}(X_g)=g$, in the sense that $X_g$ carries no
limit linear series $\mathfrak g^1_{g-1}$ and there are a finite number of $\mathfrak g^1_g$'s corresponding to the unique choice of
an integer $0\leq a\leq \frac{g}{2}$, a unique $l_{\widetilde{E}}\in G^1_g(\widetilde{E})$  with vanishing sequence $(a_0, g-a_0)$ at
both $x$ and $y$, and to a finite number of $l_{C_i}\in G^1_g(C_i)$ with vanishing sequence $(a_0, g-a_0)$ at $p_i\in C_i$ for $i=1, 2$.
\end{proof}

\section{Syzygies of sections of Nikulin surfaces}
\label{sec: etale}
In this section we study syzygies of \'etale double covers lying on Nikulin $K3$ surfaces. The moduli space $\F_g^{\mathfrak{N}}$ of Nikulin surfaces of genus $g$ has been studied in \cite{vGS} and \cite{FV} which serve as a general reference. Let us recall a few definitions. A \emph{Nikulin involution} on a smooth $K3$ surface $Y$ is a symplectic involution $\iota \in \mbox{Aut}(Y)$. A Nikulin involution has  $8$ fixed points \cite{Ni}. The quotient
$\bar{Y}:=Y/\langle \iota\rangle$ has $8$ singularities of type $A_1$. We denote by $\sigma:\tilde{S}\rightarrow Y$ the blow-up of the $8$ fixed points, by $E_1, \ldots, E_8\subset \tilde{S}$ the exceptional divisors, and finally by $\tilde{\iota}\in \mbox{Aut}(\tilde{S})$ the automorphism induced by $\iota$. Then $S:=\tilde{S}/\langle \tilde{\iota}\rangle$ is a smooth $K3$ surface. If $f:\tilde{S}\rightarrow S$ is the projection, then $N_i:=f(E_i)$ are $(-2)$-curves on $S$. The branch divisor of $f$ is equal to $N:=\sum_{i=1}^8 N_i$.
We have the following diagram that shall be used for the rest of this section:
\begin{equation}\label{diagram}
\begin{CD}
{\tilde{S}} @>{\sigma}>> {Y} \\
@V{f}VV @V{}VV \\
{S} @>{}>> {\bar{Y}} \\
\end{CD}
\end{equation}
As usual, $H^2(Y, \mathbb Z)=U^3\oplus E_8(-1)\oplus E_8(-1)$ is the unique even unimodular lattice of signature $(3, 19)$, where $U$ is the rank $2$ hyperbolic lattice and $E_8$ is the unique even, negative-definite unimodular lattice of rank $8$.
As explained in \cite{vGS}, the action of the Nikulin involution $\iota$ on the group $H^2(Y, \mathbb Z)$ is given by
$$\iota^*(u, x, y)=(u, y, x),$$ where
$u\in U$ and $x, y\in E_8(-1)$. We identify the orthogonal complement
$$\bigl(H^2(Y, \mathbb Z)^{\iota}\bigr)^\bot=\{(0, y, -y):y\in E_8(-1)\}=E_8(-2).$$
Since $\iota^*(x)=-x$ for $x\in \bigl(H^2(Y, \mathbb Z)^{\iota}\bigr)^\bot$ whereas $\iota^*(\omega)=\omega$ for $\omega\in H^{2, 0}(Y)$, it follows that $x\cdot \omega=0$, therefore $E_8(-2)\subset \mathrm{Pic}(Y)$. This shows that the Picard number of $Y$ is at least $9$.
\vskip 3pt
By construction, the class $\OO_S(N_1+\cdots+N_8)$ is even and we consider the class $e\in \mathrm{Pic}(S)$  such that $e^{\otimes 2}=\OO_S(N_1+\cdots+N_8)$.

\begin{definition} The \emph{Nikulin lattice} is an even lattice $\mathfrak{N}$ of rank $8$ generated by elements $\{\mathfrak{n}_i\}_{i=1}^8$ and $\mathfrak{e}:=\frac{1}{2}\sum_{i=1}^8 \mathfrak{n}_i$, with the bilinear form induced by $\mathfrak{n}_i^2=-2$ for $i=1, \ldots, 8$ and
$\mathfrak{n}_i\cdot \mathfrak{n}_j=0$ for $i\neq j$.
\end{definition}
Note that $\mathfrak N$ is the minimal primitive sublattice of $H^2(S, \mathbb Z)$ containing the classes $N_1, \ldots, N_8$ and $e$.
We fix $g\geq 2$ and consider the lattice $$\Lambda_g:=\mathbb Z\cdot \mathfrak c\oplus \mathfrak{N},$$
 where $\mathfrak c\cdot \mathfrak c=2g-2$. A Nikulin surface of genus $g$ is a $K3$ surface $S$ together with a primitive embedding of lattices $j:\Lambda_g\hookrightarrow \mathrm{Pic}(S)$ such that $C:=j(\mathfrak c)$ is a numerically effective class. The moduli space $\F_g^{\mathfrak{N}}$ of Nikulin surfaces of genus $g$ is an irreducible $11$-dimensional variety. Its very general point corresponds to a Nikulin surface with $\mathrm{Pic}(S)=\Lambda_g$.

\vskip 3pt
Let $f: \widetilde{S} \rightarrow  S$ be a Nikulin surface together with a smooth curve $C\subset S$ of genus $g$ such that $C\cdot N=0$.
If $\widetilde C := f^{-1}(C)$, then
$$f_C:=f_{|\widetilde C}: \tilde C \rightarrow C
$$
is an \'etale double cover induced by the torsion line bundle
$e_C:=\mathcal O_C(e)\in \mathrm{Pic}^0(C)[2]$. Thus $[C, e_C]\in \cR_g$.

Since $\widetilde{C}$ is disjoint from the $(-1)$-curves $E_i\subset \widetilde{S}$, we identify $\widetilde{C}$ with its image $\sigma(\widetilde{C})\subset Y$. Clearly $\widetilde{C}\in \bigl(E_8(-2)\bigr)^\bot$ and $(\widetilde{C})^2_Y=4(g-1)$.

One has the following result, see \cite{vGS} Proposition 2.7 and \cite{GS} Corollary 2.2,  based on a description of the map $f^*:H^2(S, \mathbb Z)\rightarrow H^2(\widetilde{S}, \mathbb Z)$:
\begin{proposition}
Let $S$ be a Nikulin surface of genus $g$ such that $j:\Lambda_g\rightarrow \mathrm{Pic}(S)$ is an isomorphism. Then $\mathbb Z\cdot \widetilde{C}\oplus E_8(-2)\subset \mathrm{Pic}(Y)$
is a sublattice of index $2$. Furthermore $E_8(-2)$ is a primitive sublattice of $\mathrm{Pic}(Y)$.
\end{proposition}
It follows that $\mbox{Pic}(Y)$ is generated by $\mathbb Z\cdot \widetilde{C}\oplus E_8(-2)$ and an element $\bigl(\frac{\widetilde{C}}{2}, \frac{v}{2}\bigr)$, where $v\in E_8(-2)$ is an element such that
$$\frac{\widetilde{C}^2}{2}+\frac{v^2}{4}\equiv 0 \mbox{ mod } 2.$$
We determine explicitly the Picard lattice of $Y$ when $\mathrm{Pic}(S)$ is minimal hence $[S, j]\in \F_g^{\mathfrak{N}}$ is a general point in moduli. The answer depends on the parity of $g$.

\begin{proposition}\label{minlattice}
Let $(S, j)$ be a Nikulin surface of genus $g$ with $\mathrm{Pic}(S)=\Lambda_g$.

\noindent (i) Suppose $g$ is odd. Then $\mathrm{Pic}(Y)$ is generated by $\mathbb Z\cdot \widetilde{C}\oplus E_8(-2)$ and an element $\bigl(\frac{\widetilde{C}}{2}, \frac{v}{2}\bigr)$, where $v^2=-8$.

\noindent (ii) Suppose $g$ is even. Then $\mathrm{Pic}(Y)$ is generated by $\mathbb Z\cdot \widetilde{C}\oplus E_8(-2)$ and an element $\bigl(\frac{\widetilde{C}}{2}, \frac{v}{2}\bigr)$, where $v^2=-4$.
\end{proposition}
\begin{proof}
The key point is that the lattice $\mathbb Z\cdot \widetilde{C}\subset \mbox{Pic}(Y)$ is primitive. This implies that if $(\frac{\widetilde{C}}{2}, \frac{v}{2})$ is the generator of $\mbox{Pic}(Y)$ over $\mathbb Z\cdot \widetilde{C}\oplus E_8(-2)$, then $v\neq 0$. The same conclusion follows directly in the case when $g$ is even for parity reasons.
\end{proof}

We are now in a position to prove that a curve $\widetilde{C}\subset Y$ corresponding to a general Nikulin surface $[S, j]\in \F_g^{\mathfrak{N}}$ satisfies Green's conjecture.
\vskip 4pt

\noindent
\emph{Proof of Theorem \ref{nik1}.} Let us choose an \'etale double cover $f:\widetilde{C}\rightarrow C$, where $\widetilde{C}\subset Y$ lies on a Nikulin surface with minimal Picard lattice and $C\subset S$.  Applying \cite{AF-Compositio},  both $\widetilde{C}$ and $C$ being sections of smooth $K3$ surfaces, satisfy Green's conjecture. It remains to determine the Clifford indices of both curves and for this purpose we resort to \cite{GL3}. First we observe that $\mbox{Cliff}(C)=[\frac{g-1}{2}]$ and the Clifford index is computed by a pencil, that is, $r(C)=1$. Indeed, otherwise $\mbox{Cliff}(C)$ is computed by the restriction to $C$ of a line bundle $\OO_S(D)$ on the surface, where $0< C\cdot D\leq g-1$. If $\mbox{Pic}(S)=\Lambda_g$, then $C\cdot D\equiv 0\mbox{ mod }2g-2$, hence no such line bundle on $S$ can exist, therefore $\mbox{Cliff}(C)$ is maximal.

Assume now that $\mbox{Cliff}(\widetilde{C})<g-1$. Since $g(\widetilde{C})=2g-1$ is odd, it follows automatically that $r(\tilde{C})=1$. Applying \cite{GL3}, there exists a divisor $D\in \mbox{Pic}(Y)$ such that $0\leq \widetilde{C}\cdot D\leq 2g-2$,
$$h^i(S, \OO_S(D))=h^i(C, \OO_{\widetilde{C}}(D))\geq 2 \ \mbox{ for } i=0, 1, \ \mbox{  and }$$
$$\mbox{Cliff}(\widetilde{C})=\mbox{Cliff}(\OO_{\widetilde{C}}(D))=\widetilde{C}\cdot D-D^2-2, $$
where the last formula follows after an application of the Riemann-Roch theorem.
Since $\widetilde{C}\in \bigl(E_8(-2)\bigr)^{\perp}$, the only class in $D\in \mathrm{Pic}(Y)$ such that $0\leq \widetilde{C}\cdot D\leq 2g-2$,
is the generator $D:=\bigl(\frac{\widetilde{C}}{2}, \frac{v}{2}\bigr)$ described in Proposition \ref{minlattice}.
When $g$ is odd we compute that
$$\widetilde{C}\cdot D-D^2-2=\widetilde{C}\cdot\Bigl(\frac{\widetilde{C}}{2}+\frac{v}{2}\Bigr)-\Bigl(\frac{\widetilde{C}}{2}+\frac{v}{2}\Bigr)^2-2=2(g-1)-(g-3)-2=g-1,$$
which contradicts the assumption $\mbox{Cliff}(\widetilde{C})<g-1$. Thus $\widetilde{C}$ has maximal Clifford index.

When $g$ is even, then $v^2=-4$. A similar calculation yields $\widetilde{C}\cdot D-D^2-2=g-2$, hence $\mbox{Cliff}(C)\geq g$. On the other hand, $\OO_{\widetilde{C}}(D)$ induces a linear series $\mathfrak g^{g/2}_{2g-2}$ on $\widetilde{C}$, which implies that $\mbox{gon}(\widetilde{C})=\mbox{Cliff}(\widetilde{C})+2=g$.

$\hfill$ $\Box$

\subsection{The Prym-Green Conjecture and Nikulin surfaces}
An analogue of Green's conjecture for  Prym-canonical curves $\phi_{K_C\otimes \eta}:C\rightarrow \PP^{g-2}$ has been formulated in \cite{FL}.

\begin{conjecture}\label{prymgreen}
 Let $[C, \eta]\in \cR_{2i+6}$ be a general Prym curve. Then $$K_{i, 2}(C, K_C\otimes \eta)=0.$$
\end{conjecture}

It is shown in \cite{FL} that the subvariety in moduli
$$\cU_{2i+6, i}:=\bigl\{[C, \eta]\in \cR_{2i+6}: K_{i, 2}(C, K_C\otimes \eta)\neq 0\bigr\}$$ is the degeneracy locus of a morphism between
two tautological vector bundles of the same rank defined over $\cR_{2i+6}$. The statement of the Prym-Green Conjecture is equivalent to the generic
non-degeneracy of this morphism. The conjecture, which is true in bounded genus, plays a decisive role in showing that the moduli space
$\rr_{2i+6}$ is a variety of general type when $i\geq 4$. The validity of Conjecture \ref{prymgreen} for unbounded $i\geq 0$ remains a challenging
open problem. In view of Voisin's solution \cite{Voisin-even}, \cite{Voisin-odd} of the classical generic Green Conjecture by specialization to
curves on $K3$ surfaces, it is an obvious question whether the Prym-Green Conjecture could be proved by specializing to Prym curves on Nikulin surfaces.
Unfortunately this is not the case, as it has been already observed in \cite{FV} Theorem 0.6. We give a second, more direct proof of the fact
 that Prym-canonical curves on Nikulin surfaces have extra syzygies.

\begin{theorem}\label{nik2}
We set $g:=2i+6$ and let $C\subset S$ be a smooth genus $g$ curve on a Nikulin surface, such that $C\cdot N=0$. Then $K_{i, 2}(C, K_C\otimes e_C)\neq 0$. In particular $[C, e_C]\in \cU_{2i+6, i}$ fails to satisfy the Prym-Green conjecture.
\end{theorem}
\begin{proof}
Since we are in a divisorial case, it is enough to prove the
nonvanishing $K_{i+1, 1}(C, K_C\otimes e_C)\neq 0$. Keeping the notation of this section, we set $H:\equiv C-e\in \mbox{Pic}(S)$. By direct calculation
$H^2=2g-6$, $H\cdot C=C^2=2g-2$ and note that $\OO_C(H)=K_C\otimes e_C$.
 The general member $H\in |\OO_S(H)|$ is a smooth curve of genus $2i+4$.
The Green-Lazarsfeld non-vanishing theorem \cite{GL-nonvanishing} applied to $H$ yields that $K_{i+1, 1}(H, K_H)\neq 0$.
Since $S$ is a regular surface, one can write  an exact sequence
$$0\longrightarrow H^0(S, \OO_S)\longrightarrow H^0(S, \OO_S(H))\longrightarrow H^0(H, K_H)\longrightarrow 0,$$
which induces an isomorphism \cite{Green-JDG84} Theorem (3.b.7)
$$\mathrm{res}_{H}:K_{i+1, 1}(S, \OO_S(H))\cong K_{i+1, 1}(H, K_H).$$ Therefore $K_{i+1, 1}(S, \OO_S(H))\neq 0$. From \cite{Green-JDG84} Theorem (3.b.1), we write the following exact sequence of Koszul cohomology groups:
$$K_{i+1, 1}(S; -C, H)\rightarrow K_{i+1, 1}(S, H)\rightarrow K_{i+1, 1}(C, H_C)\rightarrow K_{i, 2}(S; -C, H)\rightarrow\cdots. $$
The group $K_{i+1, 1}(S; -C, H)$ is by definition the kernel of the morphism
$$\wedge^{i+1}H^0(S, H)\otimes H^0(S, \OO_S(H-C))\rightarrow \wedge^i H^0(S, H)\otimes H^0(S, \OO_S(2H-C)).$$
But $H^0(S, \OO_S(H-C))=H^0(S, -e)=0$, that is, the first map in the exact sequence above is injective, hence $K_{i+1, 1}(C, \OO_C(H))\neq 0.$
\end{proof}

\section{Green's conjecture for general covers of plane curves}
\label{sec: 6-gonal}

In this section we prove the vanishing of $K_{g-2d+3,1}(C, K_C)$ for general
covers of plane curves of degree $d$. Firstly, we show that the minimal
pencils come from the plane curve.

\begin{lemma}
\label{lem: gon}
Let $f:C\rightarrow \Gamma$ be a genus $g$ double cover of a plane curve
of degree $d\ge 3$. If $g> (d-2)(d+1)$, then $C$ is $(2d-2)$-gonal.
\end{lemma}

\proof
Apply the Castelnuovo-Severi inequality, see \cite{ACGH} Chapter VIII.
\endproof

Observe that the curves in question carry infinitely many $\mathfrak{g}^1_{2d-2}$ pulled back from $\Gamma$,
hence they do not verify the linear growth condition (\ref{lgc}).

This phenomenon occurs quite often, if the genus is large
enough compared to the gonality.

\begin{proposition}
Let $C$ be a smooth curve of genus $g$ and gonality $k$ such
that $g>(k-1)^2$. If $C$ carries two different $\mathfrak g^1_k$
then there exists a cover $C\to X$ such that the
two $\mathfrak g^1_k$ are pullbacks of pencils on $X$.
\end{proposition}

\proof
We apply the Castelnuovo-Severi inequality. The two pencils
define a morphism $C\to \PP^1\times \PP^1$, and
the image is of numerical type $(k,k)$. Then the genus of
the normalization $X$ of the image is at most $(k-1)^2$, hence
$X$ cannot be isomorphic to $C$. The two rulings lifted to $X$
pullback to the original $\mathfrak g^1_k$'s on $C$.
\endproof

\begin{theorem}
\label{thm: double}
 Let $C\to \Gamma\subset\PP^2$ be a general ramified
double covering of genus
$g\ge d^2+1$ of a smooth plane curve of degree $d\ge 3$. Then $C$
verifies Green's conjecture, that is
$K_{2d-5,2}(C,K_C)=0$.
\end{theorem}

\begin{corollary}
\label{cor: double}
 Let $C\to \Gamma\subset\PP^2$ be a general ramified
double covering of genus
$g\ge 17$ of a smooth plane quartic. Then $K_{3,2}(C,K_C)=0$.
\end{corollary}

\begin{remark}
 The moduli space of double covers of
smooth plane curves of degree $d$ is irreducible,
and hence it makes sense to
speak about general double covers.
\end{remark}

\proof
From the semicontinuity
of Koszul cohomology and the irreducibility of the moduli space of double covers over smooth plane curves of degree $d$,
the conclusion follows by exhibiting one example of a
double cover $C$ of a plane curve of degree $d$, for which
$K_{2d-5,2}=0$. The proof goes by induction
on the genus $g$ of $C$, using degenerations.

\medskip

{\em The first step.}
Let $S\to \PP^2$ be a double cover ramified along
a sextic. The inverse image $C$ of a general plane curve $\Gamma$
of degree $d$ is
a $(2d-2)$-gonal smooth curve of genus $d^2+1$ (the number of
ramification points is $6d$). Applying theorem \ref{thm: AF}, it satisfies Green's
conjecture, and hence $K_{g-2d+3,1}(C,K_C)=0$.

\medskip

{\em The induction step.} Suppose that the conclusion is true
in genus $g$. We wish to prove it in genus $g+1$.
Consider $f:C\to \Gamma$ a {\em smooth} genus-$g$ double
cover of a plane curve of degree $d$, for which $K_{g-2d+3,1}(C,K_C)=0$.
Let $x\in\Gamma$ be a general point and
$\{x_0,x_1\}=f^{-1}(x)\subset C$ be the  fiber over $x$.
Attach a rational curve to $C$, gluing it over
two points $y_0,y_1\in \PP^1$ with $C$,
that is, consider
\[
 C^\prime:=C\cup \PP^1/
x_0\sim y_0,\ x_1\sim y_1.
\]
Observe that there is an admissible double cover $C^\prime\to \Gamma^\prime$,
where $\Gamma^\prime = \Gamma\cup \PP^1/x\sim y$,
where $y\in \PP^1$, see the figure \ref{fig:cover}.
\begin{figure}[htb!]
\centering%
\includegraphics[width=5.5cm, height=4.2cm]{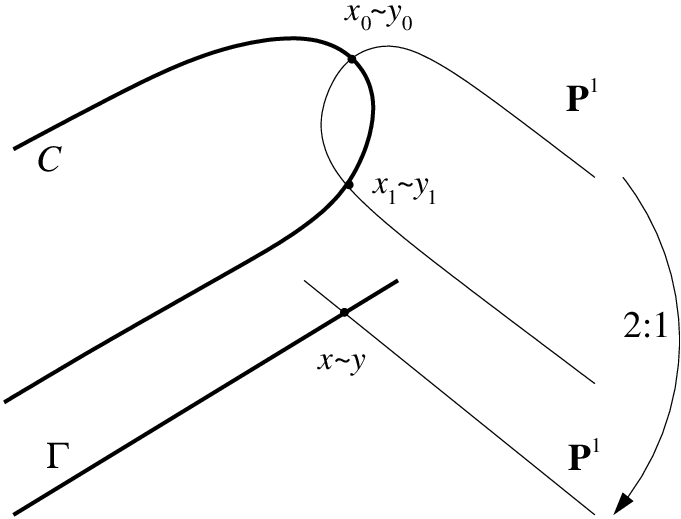}
\caption{The new admissible double cover.}
\label{fig:cover}
\end{figure}

It is clear that the genus of $C^\prime$ equals $g+1$ and
$p_a(\Gamma^\prime)=p_a(\Gamma)$. Arguing
as in \cite{Voisin-even}, the restriction map provides us
with an isomorphism
\[
 K_{p,1}(C^\prime,\omega_{C^\prime})\cong K_{p,1}(C,K_C(x_0+x_1)).
\]
From the induction hypothesis we know that
$K_{g-2d+3,1}(C,K_C)=0$. Applying Corollary \ref{cor: canonical}, it follows that
$K_{g-2d+4,1}(C,K_C(x_0+x_1))=0$, hence
$$K_{(g+1)-2d+3,1}(C^\prime,\omega_{C^\prime})=0,$$
the latter being the vanishing we wanted to obtain.
\endproof

\vskip 3pt
\noindent \emph{Proof of the second part of Theorem \ref{doubleplane}.} This time we start with a $K3$ surface
$S$ which is a cyclic fourfold cover of $\PP^2$ branched along a quartic. The inverse image of a general plane curve of degree $d$ is a curve $C$ with $g(C)=2d^2+1$ and $\mbox{gon}(C)=4d-4$. The induction step is similar to the one in Theorem \ref{thm: double} see figure \ref{fig:cover-4-1}.
\hfill $\Box$

\begin{figure}[htb!]
\centering%
\includegraphics[width=7.5cm, height=6cm]{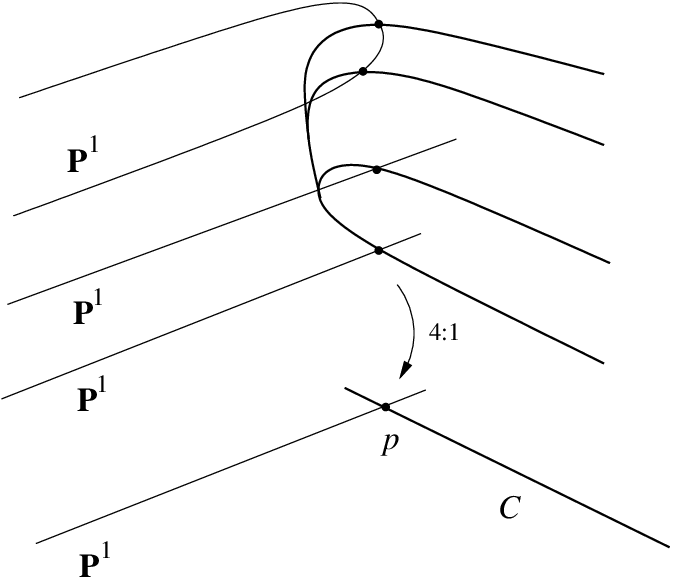}
\caption{The new admissible $4:1$ cover.}
\label{fig:cover-4-1}
\end{figure}

The curves on the double plane that we use in the
first step of the proof carry infinitely many minimal pencils,
and hence they do not verify the linear growth condition
(\ref{lgc}).
They are in fact special in their linear systems. According
to \cite{AF-Compositio}, a general curve in the corresponding
linear system does satisfy the linear growth condition.
This provides us with an example of a linear system on a $K3$ surface
where the dimensions of the Brill-Noether loci jump.
%However it makes sense to ask the following question:

%\begin{question}
%Is it true that any smooth curve of even genus and maximal
%Clifford index on a $K3$ surface carries finitely many minimal pencils?
%\end{question}

\section{Green's conjecture for general triple covers of elliptic curves}

Applying the Castelnuovo-Severi inequality as in Lemma \ref{lem: gon}, we obtain that if $C\rightarrow E$ is a triple cover of an elliptic curve $E$, then $C$ is $6$-gonal as soon as $g(C)\geq 12$.

\begin{theorem}
\label{thm: triple 2}
Let $C\to E$ be a general triple cover of an elliptic curve, where $g(C)\geq 13$. Then $K_{3,2}(C,K_C)=0$
and $C$ verifies Green's Conjecture.
\end{theorem}

\proof
The proof goes by induction on the genus and is very similar to that
of Theorem \ref{thm: double}. Note that the moduli space of
triple covers of elliptic curves is irreducible by e.g. \cite{GHS},
hence it suffices to find an example in each genus.

\medskip

{\em The first step.}
Let $S\to \PP^1\times \PP^1:=Q$ be a cyclic triple
cover ramified along a smooth genus $4$ curve, which has type $(3, 3)$ on $Q$. It is immediate that
$S$ is a $K3$ surface.
The inverse image $C$ of a general curve $E$
of type $(2,2)$ is
a smooth $6$-gonal curve of genus $13$, and
the induced triple cover $C\to \Gamma$ is ramified
over $24$ points (the ramification points of a cyclic cover
are totally ramified, thus the degree of the ramification divisor is $48$).
Since $S$ is a $K3$ surface, we apply \cite{AF-Compositio},
to conclude that $K_{3,2}(C,K_C)=0$.

\medskip

{\em The induction step.} We suppose that the conclusion is true
in genus $g$ and we prove it in genus $g+1$.
Consider a triple covering $f:C\to E$, where both $C$ and $E$ are smooth curves, $g(C)=g\geq 13$ and $g(E)=1$. Assume that
$K_{g-5,1}(C,K_C)=0$.
Let $t\in E$ be a non-ramified point and
$\{x_0,x_1, x_2\}=f^{-1}(t)\subset C$ be the  fiber over $t$.
Attach a rational curve $R$ to $C$, gluing it along $x_0$ and $x_1$, as well as a further rational tail $R'$ meeting $C$ in $x_2$,
that is, consider the (non)-stable curve
\[
 C^\prime:=C\cup R\cup R', \ \ C\cap R=\{x_0, x_1\}, \ \ C\cap R'=\{x_2\}.
\]
There exists an admissible triple cover $f':C^\prime\to E^\prime$,
where $E^\prime = \Gamma\cup_t \PP^1$, where $f'(R)=f'(R')=\PP^1$, $\mbox{deg}(f'_R)=2$ and $\mbox{deg}(f'_{R'})=1$.

The genus of $C^\prime$ equals $g+1$ and  there is an isomorphism
\[
 K_{p,1}(C^\prime,\omega_{C^\prime})\cong K_{p,1}(C,K_C(x_0+x_1)).
\]
From the induction hypothesis we know that
$K_{g-5,1}(C,K_C)=0$. Applying projection of
syzygies, it follows that
$K_{g-4,1}(C,K_C(x_0+x_1))=0$, hence
$K_{(g+1)-5,1}(C^\prime,\omega_{C^\prime})=0$,
the latter being the vanishing we were looking for.
\endproof

\begin{remark}
A slight modification in the proof shows that Green's Conjecture also holds for general \emph{cyclic triple covers} of elliptic curves with source being a curve of
odd genus $g\geq 13$. The modification of the proof appears in the inductive argument. Starting with $f:C\rightarrow E$ as above, we can attach
a smooth rational curve meeting $C$ at $x_0, x_1$ and $x_2$. The resulting curve has genus $g+2$ and smooths to a cyclic cover over an elliptic curve.
\end{remark}

\section{Syzygies of double covers of curves of Clifford dimension $3$}
We present an inductive proof of Theorem \ref{cliffdim3} and consider a curve $[X]\in \cM_{10}$ with $r(X)=3$, thus $W^3_9(C)\neq \emptyset$ and $\mbox{dim } W^1_6(X)=1$.
If $f:C\rightarrow X$ is a genus $g$ double cover, the Castelnuovo-Severi inequality implies that $\mbox{gon}(C)=12$ as soon as $g\geq 30$.
The critical point in the proof is the starting case, the inductive step is identical to that in the proof of Theorem
\ref{doubleplane}.
\vskip 3pt

\noindent \emph{Proof of Theorem \ref{cliffdim3}.} We choose a smooth cubic surface $Y=\mbox{Bl}_6(\PP^2)$ and denote by
$h\in \mbox{Pic}(S)$ the class of the pull-back of a line in $\PP^2$
and by $E_1, \ldots, E_6$ the exceptional divisors on $Y$. We choose a general genus $4$ curve
$$B\in |-2K_Y|=|\OO_Y(6h-2E_1-\cdots-2E_6)|$$
and let $f:S\rightarrow Y$ be the double cover branched along $B$. Then
$S$ is a smooth $K3$ surface and let $\iota\in \mbox{Aut}(S)$ be the covering
involution of $f$.
Clearly $H^2(S, \mathbb Z)^{\iota}$ can be identified with the pull-back of the Picard lattice of
$Y$, and when $B\in |-2K_Y|$ is general, reasoning along the lines of \cite{AK} Theorem 2.7
we observe that
$$\mbox{Pic}(S)=H^2(S, \mathbb Z)^{\iota}=f^*\mbox{Pic}(Y)=\mathbb Z\langle f^*(h), \OO_S(R_1), \ldots, \OO_S(R_6)\rangle,$$ where $R_i:=f^*(E_i)$ are $(-2)$-curves.
We further choose a general curve $X\in |-3K_Y|$, thus $g(X)=10$ and $r(X)=3$.
Let $C:=f^{-1}(X)\subset S$, hence $g(C)=28$. As a section of the $K3$ surface $S$, the curve $C$ satisfies Green's Conjecture and Theorem \ref{cliffdim3} follows once we show that
$\mbox{gon}(C)=12$. Assume by contradiction that $\mbox{gon}(C)<12$. Applying once more \cite{GL3}, there exists a divisor class
$$D\equiv af^*(h)-b_1R_1-\cdots-b_6R_6\in \mathrm{Pic}(S),$$
with $a, b_1, \ldots, b_6\in \mathbb Z$, such that $0\leq C\cdot D\leq g-1=27$, $h^i(S, \OO_S(D))\geq 2$ for $i=0, 1$ and
$$\mbox{gon}(C)=\mbox{Cliff}(\OO_C(D))+2=C\cdot D-D^2=$$
$$=\phi(D):=18a-2a^2-6(b_1+\cdots+b_6)+2(b_1^2+\cdots+b_6^2)<12.$$
From the Castelnuovo-Severi inequality, we find that $\phi(D)\geq 9$, hence based on parity $\phi(D)=10$.
Note that $C\cdot D\ge 10$ and is a multiple of $6$, hence
$C\cdot D\in \{12, 18, 24\}$.
We study each of these cases separately. If $C\cdot D=18$ and $D^2=8$, then
$$b_1+\cdots+b_6=3a-3 \mbox{ and } b_1^2+\cdots+b_6^2=a^2-4.$$

By the Cauchy-Schwarz inequality $6(\sum_{i=1}^6 b_i^2)\geq (\sum_{i=1}^6 b_i)^2$, and
hence $a^2-6a+11\leq 0$, which is a contradiction. If $C\cdot D=24$ and $D^2=14$, then
$$b_1+\cdots+b_6=3a-4 \mbox{ and } b_1^2+\cdots+b_6^2=a^2-7,$$
which leads to the contradiction $3a^2-24a+58\leq 0$. Finally if $C\cdot D=12$ and $D^2=2$, then
$$b_1+\cdots+b_6=3a-2 \mbox{ and }\ b_1^2+\cdots+b_6^2=a^2-1.$$
Again the Cauchy-Schwarz inequality implies that the only possible case is when $a=2$ and then $\sum_{i=1}^6 b_i=4$ and $\sum_{i=1}^6 b_i^2=3$. It is obvious (compare the parities)
that these diophantine equations have no common solution. We conclude that $\mbox{gon}(C)=12$.

$\hfill$ $\Box$

% \vskip 3pt
% \noindent \tiny{{\bf{Acknowledgements}}: The first author was supported in part by the SFB 647 Raum-Zeit-Materie, by
% a PN-II-ID-PCE-2008-2 grant, PCE-2228 (contract no. 502/2009) and
% by LEA MathMode. He warmly thanks HU Berlin, IEC Nancy and Institut Poincar\'e Paris
% for hospitality. The second author was supported by the SFB 647 Raum-Zeit-Materie and by the Schwerpunktprogramm of the DFG Algorithmische Methoden in Algebra und Zahlentheorie.}

\end{document}